\documentclass{iopconfser}
\usepackage{amsmath,amssymb,amsthm,latexsym,amsbsy,euscript}

\newtheorem{theor}{Theorem}

\theoremstyle{definition}
\newtheorem{define}{Definition}

\newtheorem{state}[theor]{Proposition}
\newtheorem{prop}[theor]{Proposition}
\newtheorem{cor}[theor]{Corollary}

\newtheorem{ex}{Example}
\theoremstyle{remark}
\newtheorem{rem}{Remark}

\newcommand{\BBR}{\mathbb{R}}\newcommand{\BBC}{\mathbb{C}}
\newcommand{\BBN}{\mathbb{N}}
\newcommand{\BBS}{\mathbb{S}}
\newcommand{\BBZ}{\mathbb{Z}}\newcommand{\BBE}{\mathbb{E}}

\newcommand{\cA}{{{\EuScript A}}}

\newcommand{\dd}{\partial}
\newcommand{\Id}{{\mathrm d}}
\newcommand{\ID}{{\mathrm D}}

\DeclareMathOperator{\Alt}{Alt}
\DeclareMathOperator{\Hom}{Hom}
\newcommand{\W}[3]{\mathop{{W}}\left(#1,#2,#3\right)}

\newcommand{\by}[1]{\textrm{{#1}}}
\newcommand{\jour}[1]{\textit{{#1}}}
\newcommand{\vol}[1]{\textbf{{#1}}}
\newcommand{\book}[1]{\textit{{#1}}}

\begin{document}\large \pagestyle{plain}

\title{Wronskians as 
$N$-ary brackets in finite\/-\/dimensional analogues of~$\mathfrak{sl}(2)$}

\author{Arthemy V Kiselev}

\affil{Bernoulli Institute for Mathematics, Computer Science and Artificial Intelligence, University of Groningen, P.O. Box~407, 9700 AK Groningen, The Netherlands}

\email{a.v.kiselev@rug.nl}

\begin{abstract}\noindent%
The Wronskian determinants (for coefficients of higher\/-\/order differential operators on the affine real line or circle
) satisfy the table of Jacobi\/-\/type quadratic identities for strong homotopy Lie algebras --\,i.e.\ for a particular case of $L_\infty$-\/deformations\,-- for the Lie algebra of vector fields on that one\/-\/dimensional affine manifold.
We show that the standard realisation of~$\mathfrak{sl}(2)$ by quadratic\/-\/coefficient vector fields is the bottom structure in a sequence of finite\/-\/dimensional polynomial algebras~$\Bbbk_N[x]$ with the Wronskians as $N$-ary brackets; the structure constants are calculated explicitly.\\[0.5pt]
\end{abstract}

\large
\section{Introduction}
\large
Let us view Lie algebras $\mathfrak{g}=(V, [{\cdot},{\cdot}])$ as a base class of structures which we seek to generalise in a natural way:
(\textit{i}) the vector space~$V$ (over a chosen field~$\Bbbk=\BBR$ or~$\BBC$ of characteristic zero) can be enlarged;
(\textit{ii}) the binary Lie bracket $[{\cdot},{\cdot}]$ can be replaced with $N$-ary antisymmetric multi\/-\/linear bracket(s) satisfying (collections of)
(\textit{iii}) suitable variants of $N$-ary Jacobi\/-\/type quadratic identities.
In retrospect, the problem and process of such enlargement sheds more light on the nature and properties of the initially taken objects and structures.

\begin{ex}\label{ExSimpleLie}
The Lie algebra $\mathfrak{sl}(2,\BBC)$ is the main example of a (semi)simple complex Lie algebra; its structure is encoded by the root system~$\mathsf{A}_1$ in~$\BBE^1$. In the Chevalley basis $\{e,f,h\}$ for~$\mathfrak{sl}(2)$, the Lie bracket is determined by the relations
\begin{equation}\label{EqRelSL2}
[h,e]=2e,\qquad [h,f]=-2f,\qquad [e,f]=h.
\end{equation}
Larger irreducible root systems in higher\/-\/dimensional real Euclidean spaces~$\BBE^r$ exhaustively classify all 
simple complex Lie algebras of higher ranks~$r$; their list is $\mathsf{A}_r$ ($r\geqslant 1$), 
$\mathsf{B}_r$ ($r\geqslant 2$), $\mathsf{C}_r$ ($r\geqslant 3$), 
$\mathsf{D}_r$ ($r\geqslant 4$), and the exceptional five: $E_6$, $E_7$, $E_8$, $F_4$, and~$G_2$, see~\cite{Humphreys}; 
allowing imaginary simple root vectors, we arrive 
at the Borcherds\/--\/Kac\/--\/Moody~algebras.
\end{ex}
   
Independently, the binary 
bracket $[{\cdot},{\cdot}]$ can be deformed to a formal series of structures,
\[
[{\cdot},{\cdot}] \longmapsto \nabla = [{\cdot},{\cdot}] + \nabla_3 +\ldots+\nabla_m+\ldots,
\]
where $[{\cdot},{\cdot}] \equiv \nabla_2$ and for each $m\geqslant 2$, the $m$-linear term~$\nabla_m$ is totally antisymmetric w.r.t. its $m$~arguments from the underlying vector space~$V$. 
The original Lie bracket~$\nabla_2$ in the algebra $\mathfrak{g}=(V,\nabla_2)$ satisfied the Jacobi identity,
\[
\tfrac{1}{1!\cdot 2!} \sum_{\tau\in S_3} (-)^\tau\: \nabla_2\bigl( \nabla_2 \bigl( v_{\tau(1)}, v_{\tau(2)} \bigr), v_{\tau(3)} \bigr) = 0,
\]
for any elements $v_1,v_2,v_3\in V$.
A natural quadratic (w.r.t.\ the structure~$\nabla$) Jacobi\/-\/type identity for the deformed structure $\nabla=\nabla_2+\ldots$ is $\nabla[\nabla]=0$,
meaning that for every $m$-tuple $v_1\otimes v_2\otimes\ldots\otimes v_m$ of $v_j\in V$ the expansion of inner-{} and outer copy of~$\nabla$ by linearity over~$\Bbbk$ produces the chain of partial (at $m\geqslant 3$) identities of the form
\begin{multline*}
\sum_{\tau\in S_m} (-)^\tau \Bigl[ 
  \nabla_2\bigl( \nabla_{m-1} \bigl( v_{\tau(1)},\ldots,v_{\tau(m-1)}\bigr),v_{\tau(m)} \bigr) + \ldots {} \\
  {}+ \nabla_k\bigl( \nabla_{m+1-k} \bigl( v_{\tau(1)},\ldots,v_{\tau(m+1-k)}\bigr),v_{\tau(m+2-k)},\ldots, v_{\tau(m)} \bigr) + \ldots {} \\
  {}+ \nabla_{m-1}\bigl( \nabla_{2} \bigl( v_{\tau(1)},v_{\tau(2)}\bigr),
v_{\tau(3)},\ldots, v_{\tau(m)} \bigr)
  \Bigr] = 0.
\end{multline*}
This is the infinite chain (as integer $m$~starts at~$3$ and increases) of Jacobi identities for the $L_\infty$-deformation of the bracket $[{\cdot},{\cdot}]=\nabla_2$ in the Lie algebra $\mathfrak{g}=(V,[{\cdot},{\cdot}])$, see~\cite{SchlessingerStasheff1985} and~\cite{KontsevichSoibelmanAinfty}.
Of particular interest is the case when --\,in each $m$th summand of the above chain of identities\,-- all the quadratic terms vanish separately, i.e.\ whenever the $N$-ary operations~$\nabla_N$ at~$N\geqslant 2$ in the $L_\infty$-structure~$\nabla$ are such that $\nabla_k[\nabla_\ell]=0$ for all $k,\ell\geqslant 2$.
We shall study these strong homotopy deformations of the Lie bracket $[{\cdot},{\cdot}]$, see~\cite{LadaStasheff1993}, the tail components~$\nabla_j$ now satisfying the table of identities (at $k,\ell\geqslant 2$),
\[
\sum_{\tau\in S_{\ell,k-1}} (-)^\tau\: 
  \nabla_k\bigl( \nabla_{\ell} \bigl( v_{\tau(1)},\ldots,v_{\tau(\ell)}\bigr),v_{\tau(\ell+1)},\ldots, v_{\tau(k+\ell-1)} \bigr) = 0,
\]
where the sums are conveniently taken over the sets 
of $(\ell,k-1)$-unshuffles $\tau\in S_{k+\ell-1}$ such that $\tau(1) <\ldots< \tau(\ell)$ and $\tau(\ell+1)<\ldots<\tau(k+\ell-1)$; passing from the entire group of permulations $S_{k+\ell-1}$ to its subset 
of unshuffles, we divide both sides of the identity
$\nabla_k[\nabla_\ell]=0$ by $(k-1)!\ell!$ occurring from the alternation of arguments $v_j$ strictly within the totally antisymmetric brackets~$\nabla_k$ and~$\nabla_\ell$, respectively.

\smallskip
{\textbf{Research problem.}}\quad
We are interested in finding a natural source of strong homotopy Lie structures $\nabla_k$, $k\geqslant2$, that would deform 
the Lie algebra~$\mathfrak{sl}(2)$. Secondly, we want to find a class of finite\/-\/dimensional vector spaces~$V_N$ such that at each $N\geqslant 2$, the $N$-ary bracket~$\nabla_N$ does restrict to~$V_N$, making it a finite\/-\/dimensional Schlessinger\/--\/Stasheff Lie algebra~$(V_N,\nabla_N)$.

\smallskip
To this end, let us consider the quadratic\/-\/coefficient realisation
$\varrho\colon \mathfrak{sl}(2) \to \ID_1(\BBR)$ of the Lie algebra $\mathfrak{sl}(2)$ in the space of vector fields on the line~$\BBR$ with global affine coordinate~$x$; this standard realisation is given by the formula\footnote{\label{FootVFtoCourant}%
Independently from our construction of two classes 
of strong homotopy Lie algebras, one can start from this vector field realisation of~$\mathfrak{sl}(2)$ in~$\Gamma(T\BBR^1)$ --\,or of higher\/-\/rank semisimple Lie algebras\,-- and study their enlargements to Courant algebra structures on $\Gamma(TM\oplus T^*M)$ over base manifolds~$M$ (e.g., $M=\BBS^1$): the Lie bracket $[{\cdot},{\cdot}]$ of vector fields is then supplemented with the new rules to commute vector fields with differential $1$-forms and similarly, commute two differential $1$-form arguments, see~\cite{CourantAlgebroids} and references therein.}
\begin{equation}\label{EqQuadraticCoeffsSL2}
\varrho(e) = 1\cdot \dd/\dd x,\qquad
\varrho(h) = -2x\cdot \dd/\dd x,\qquad
\varrho(f) = -x^2\cdot \dd/\dd x.
\end{equation}
One readily verifies the standard commutation relations from Eq.~\eqref{EqRelSL2}; our choice of sign in the commutator is $\bigl[\vec{X},\vec{Y}\bigr] = \bigl[\vec{X}(Y) - \vec{Y}(X) \bigr]\cdot \dd/\dd x$ for $\vec{X} = X(x)\cdot \dd/\dd x$ and $\vec{Y} = Y(x)\cdot\dd/\dd x$.
In the commutator $[{\cdot},{\cdot}]$ on~$\ID_1(\BBR^1)$ we recognise the Wronskian determinant of two coefficients:
\[
\bigl[\vec{X},\vec{Y}\bigr](x) = \det\begin{pmatrix} X & Y \\ X' & Y' \end{pmatrix}(x)\cdot \frac{\dd}{\dd x} = W^{0,1}(X,Y)(x)\cdot \frac{\dd}{\dd x}.
\]
Of course, the commutator $[{\cdot},{\cdot}]$ of vector fields does satisfy the Jacobi identity,
\[
\tfrac{1}{2} \sum_{\tau\in S_3} (-)^\tau\: W^{0,1}\bigl(
  W^{0,1}(X_{\tau(1)},X_{\tau(2)}), X_{\tau(3)} \bigr)\cdot \dd/\dd x=0,
\]
for any vector fields $\vec{X}_j = X_j(x)\,\dd/\dd x$ with twice defferentiable coefficients~$X_j(x)$ on~$\BBR$.

Viewing the Wronskian determinants as ($N\geqslant 2$)-ary brackets for coefficients of higher\/-\/order differential operators on the affine line, we presently describe a class of finite\/-\/dimensional Schlessinger\/--\/Stasheff Lie algebras --- such that this class incorporates the standard realisation $\varrho\colon \mathfrak{sl}(2) \to \ID_1(\BBR)$ with quadratic polynomials in Eq.~\eqref{EqQuadraticCoeffsSL2}.

\section{Wronskians as $N$-ary brackets}\label{SecWronsk1DBrackets}\noindent%
Consider the associative algebra $\ID_*(\BBS^1)$ of differential operators of nonnegative integer orders on the circle~$\BBS^1$ (or on the other connected one\/-\/dimensional affine real manifold $M^1=\BBR^1_{\text{aff}}$). The assumption that allowed coordinate transformations are affine makes well defined the subspaces $\ID_p(M^1)$ of differential operators of strict order~$p$.

\begin{define}\label{DefAltNBr}
Take $N=2p$ and let $w_1,\ldots,w_N\in\ID_p(M^1)$, so that locally we have $w_j=w_j(x)\cdot\dd_x^p$. By definition, put
\begin{equation}\label{EqAltBrOp}
\bigl[w_1,\ldots,w_N\bigr]_N \mathrel{{:}{=}} \Alt\bigl(w_1,\ldots,w_N\bigr) =
  \sum_{\sigma\in S_N} (-)^\sigma\: w_{\sigma(1)}\circ \ldots \circ w_{\sigma(N)},
\end{equation}
where the r.-h.s.\ is the alternated associative composition of operators.
\end{define}

In the same way as the commutator of two vector fields is again a vector field, we prove

\begin{theor}\label{ThOrderPreserved}
The subspace $\ID_p(M^1)$ of differential operators $w_j=w_j(x)\cdot \dd_x^p$ of strict order~$p\in\BBN$ is closed under the alternated composition $[{\cdot},\ldots,{\cdot}]_{N=2p}$ of twice as many arguments $w_1,\ldots,w_N\in\ID_p(M^1)$.
Moreover, the structure constants are explicit\textup{:}
\begin{equation}\label{EqBrWr}
\bigl[w_1(x)\,\dd_x^p,\ldots,w_N(x)\,\dd_x^p\bigr]_N =
  W^{0,1,\ldots,N-1}(w_1,\ldots,w_N)\cdot \dd_x^p,
\end{equation}
where $W^{0,1,\ldots,N-1} = \mathbf{1}\wedge \dd_x\wedge\ldots\wedge \dd_x^{N-1}$ is the Wronskian determinant of $N$~arguments in one independent variable~$x$.
\end{theor}

\begin{rem}\label{RemWronskBehaves}
The Wronskian determinant of $N$~scalar functions itself is \textsl{not} a scalar function: indeed, the Wronskian determinant behaves under a change of base coordinate, $x=x(y)$ ${}{\rightleftarrows}{}$ $y=y(x)$, locally on~$M^1$ (see App.~\ref{AppWronsk}).
Likewise, the coefficients of differential operators of order $p>0$ do change after a reparametrisation on the base~$M^1$; when this change is affine, the strict order~$p$ is preserved and equality~\eqref{EqBrWr} makes sense.
\end{rem}


We deduce from Theorem~\ref{ThOrderPreserved} that the alternated composition of $N=4$ differential operators of order $p=2$ is again an operator of order two. The same holds for every integer $N=2p$; this serves an $N$-ary generalisation for the commutator of vector fields from~$\ID_1(M^1)$. Still, given the alternated composition as the bracket for elements of an associative algebra, which quadratic, Jacobi\/-\/type identities does this bracket satisfy\,?

\begin{prop}[\cite{DzhFAP,HanlonWachs}]\label{PropEvenN}
Let $\cA$~be an associative algebra and $[{\cdot},\ldots,{\cdot}]_N\in\Hom_\Bbbk\bigl(
\bigwedge^N \cA,\cA\bigr)$ be the alternated composition of $N$~elements $a_1,\ldots,a_N$ from~$\cA$ (cf.\ Eq.~\eqref{EqAltBrOp}): 
\begin{equation}\label{AssocBracketHW}
\bigl[a_1,\ldots,a_N\bigr]_N =
\sum_{\sigma\in S_N} (-)^\sigma\: a_{\sigma(1)}\circ\ldots\circ a_{\sigma(N)}.
\end{equation}
Suppose also that $N$ is even.
Then the bracket $[{\cdot},\ldots,{\cdot}]_N$ satisfies the quadratic Jacobi\/-\/type identity,
\[
\frac{1}{N!(N-1)!}\: \sum_{\tau\in S_{2N-1}} (-)^\tau\:
  \bigl[ \bigl[ a_{\tau(1)}, \ldots, a_{\tau(N)} \bigr]_N, a_{\tau(N+1)}, \ldots,
     a_{\tau(2N-1)} \bigr]_N = 0,
\]
so that $\cA$~becomes a Schlessinger\/--\/Stasheff Lie algebra.
\end{prop}

The proof is by inspecting the coefficient of $a_1\circ a_2\circ\ldots\circ a_{2N-1}$ in the totally antisymmetric sum over $S_{2N-1}\ni\tau$; whenever $N$~is even, the coefficient cancels~out.

\begin{cor}\label{CorNevenSHLie}
For even $N=2p\in\BBN$, the Wronskian determinant $W^{0,1,\ldots,N-1} = \mathbf{1}\wedge \dd_x\wedge\ldots\wedge \dd_x^{N-1}$ over a one\/-\/dimensional base~$M^1$ satisfies the $N$-ary Jacobi identity $W^{0,1,\ldots,N-1} \bigl[ W^{0,1,\ldots,N-1} \bigr] = 0$.
\end{cor}

In the course of proving Proposition~\ref{PropEvenN} it is readily seen that its idea extends to a not necessarily even number of arguments in either inner-{} or outer bracket and to a not necessarily coinciding number of arguments in the inner-{} and outer brackets within the left\/-\/hand side of the quadratic Jacobi identity for strong homotopy Lie algebras.

\begin{state}[\cite{DzhFAP}]\label{PropLinKout}
Recall that the subscript~$N$ at the symbol~$\Delta_N$ of bracket~\eqref{AssocBracketHW} denotes its number of arguments\textup{:}
$\Delta_i\in\Hom_\Bbbk(\bigwedge^N\cA,\cA)$\textup{;} let $k$ and $\ell$ be
arbitrary positive integers. Then the following identities hold\textup{:}
\begin{subequations}
\begin{align}
\Delta_{2k}[\Delta_{2\ell}]&=0,\label{BothEven}\\
\Delta_{2k+1}[\Delta_{2\ell}]&=\Delta_{2k+2\ell},\label{InnerEven}\\
\Delta_k[\Delta_{2\ell+1}]&=k\cdot\Delta_{2\ell+k}.\label{InnerOdd}
\end{align}
\end{subequations}
\end{state}
\begin{proof}
The proof of~\eqref{BothEven} repeats literally the 
proof of Proposition~\ref{PropEvenN}.
For~\eqref{InnerEven}, we note that the last summand, 
\[
\beta_{2k+1} = (-)^{2\ell \cdot((2k+1)-1)} \Delta_{2k+1}\bigl(
  \Delta_{2\ell}\bigr(a_{2k+1},\ldots,a_{2k+2\ell}\bigr), a_1,\ldots,a_{2k}\bigr),
\]
is not compensated. 
For \eqref{InnerOdd}, the summand $\alpha=a_1\circ\dots\circ a_{2\ell+k}$ acquires the
coefficient
$
\sum_{j=1}^k(-1)^{(2\ell+1)(j-1)}\cdot(-1)^{j-1}=k
$.
This completes the proof.
\end{proof}

These properties of totally antisymmetric homomorphisms
work immediately for Wronskian determinants of arbitrary and not necessarily coinciding sizes.

\begin{prop}[see~\cite{Dzhuma2002}]\label{PropJacobiWanyWsize}
Consider the Wronskian determinants 
$W^{0,1,\ldots,N} = \mathbf{1}\wedge \dd_x\wedge\ldots\wedge \dd_x^{N}$
with integral orders of differentiation.
Then the strong homotopy Lie algebra Jacobi identities
$W^{0,1,\ldots,k} \bigl[ W^{0,1,\ldots,\ell} \bigr] = 0$ hold for all positive integers $k,\ell\in\BBN$, meaning that
\begin{equation}\label{EqSHJacobiKL}
\frac{1}{k!(\ell+1)!}\: \sum_{\tau\in S_{k+\ell+1}} (-)^\tau\:
W^{0,1,\ldots,k}\bigl( W^{0,1,\ldots,\ell} \bigl( f_{\tau(1)}, \ldots, f_{\tau(\ell+1)}
  \bigr), f_{\tau(\ell+2)}, \ldots, f_{\tau(k+\ell+1)} \bigr) = 0
\end{equation}
for arbitrary $f_{1}(x),\ldots,f_{k+\ell+1}(x)$ of one independent variable~$x$.
\end{prop}

\begin{proof}
Indeed, the inner-{} and outer Wronskian determinants combined contain $\tfrac{1}{2} k(k+1) + \tfrac{1}{2} \ell(\ell+1)$ derivatives~$\dd_x$ acting on the arguments of Jacobiator; by construction, the Jacobiator $W^{0,1,\ldots,k} \bigl[ W^{0,1,\ldots,\ell} \bigr]$ is totally antisymmetric w.r.t.\ its $k+\ell+1$ arguments. For this, to let the integral differential orders of all the arguments~$f_j$ be pairwise distinct, at least $\tfrac{1}{2} (k+\ell+1)(k+\ell+2)$ derivatives are needed. But the actually available number is strictly less, whence the assertion.\footnote{\label{FootClaimDdim}
This proof of the claim about Wronskians over one\/-\/dimensional base manifolds~$M^1$ does extend to a properly defined class of Wronskian determinants for arguments in $d$~variables $x^1,\ldots,x^d$, see~\cite{ForKac}.}
\end{proof}

\begin{rem}\label{RemLieVFsameNN}
Although the Wronskian determinant of size $2\times 2$ does show up in the commutator of vector fields, their differential order $p=1$ is too low to make noticeable that strong homotopy Jacobi identities~\eqref{EqSHJacobiKL} are valid for not necessarily equal numbers of arguments in the inner-{} and outer brackets.
\end{rem}

\begin{rem}\label{RemWhoseCoeffs}
Whenever the number~$N$ of arguments in the Jacobiator $W^{0,1,\ldots,N-1} \bigl[ W^{0,1,\ldots,N-1} \bigr]$ is not even, we no longer refer to the arguments, depending on the variable~$x$, as coefficients of differential operators of strict (half-)\/integer order $p=N/2$. Indeed, there is presently no guarantee that the alternated composition of half\/-\/integral order operators would act by \textsl{integer} order differentiations, $\mathbf{1}\wedge \dd_x\wedge\ldots\wedge \dd_x^{N-1} = W^{0,1,\ldots,N-1}$, on such operators' coefficients~$f_j(x)$.
\end{rem}

\section{Finite\/-\/dimensional algebras $\Bbbk_N[x]$ with Wronskian brackets}\label{SecWronsk1DFiniteDim}\noindent%
The quadratic\/-\/polynomial realisation of three\/-\/dimensional Lie algebra $\mathfrak{sl}(2)$ can carry the \textsl{ternary} Wronskian bracket $W^{0,1,2}$ and be closed w.r.t.~it. Are there larger, still finite\/-\/dimensional Schlessinger\/--\/Stasheff $N$-ary Lie algebras of polynomials\,?

Consider the space $\Bbbk_N[x]\ni a_j$ of polynomials 
of degree not
greater than~$N$; on this space, the Wronskian determinant is an $N$-linear antisymmetric bracket,
\begin{equation}\label{NaryBracket}
\bigl[a_1, \ldots, a_N\bigr]_N = W^{0,1,\ldots,N-1}\bigl({a_1},{\ldots},{a_N}\bigr).
\end{equation}
Introduce the basis $\{a^0_k\}=\{x^k / k! \}$ of monomials in~$\Bbbk_N[x]$, here $0\leqslant k\leqslant N$;
the mo\-no\-mi\-als $x^k/k!$ 
are closed w.r.t.~
derivations --- and the Wronskian determinants as~well. 


\begin{theor}\label{ThDivideBySkip}
Let $0\leqslant k\leqslant N$ and bypass the monomial $x^k/k!$ from our basis in~$\Bbbk_N[x]$. Then the Wronskian determinant of remaining monomials satisfies the identity
\begin{equation}\label{WronskWithConst}
W^{0,1,\ldots,N-1}\Bigl({1},{\ldots,\widehat{\frac{x^k}{k!}},\ldots},{\frac{x^N}{N!}}\Bigr)= \frac{x^{N-k}}{(N-k)!}.
\end{equation}
In particular, all the structure constants, whenever nonzero, equal~$\pm 1$ in this $(N+1)$-\/dimensional Schlessinger\/--\/Stasheff Lie algebra~$\Bbbk_N[x]$ with the Wronskian as $N$-ary bracket. 
\end{theor}

\begin{proof}
We have
\begin{equation}\label{WronskDecompose}
\W{1}{\ldots,\widehat{\frac{x^k}{k!}},\ldots}{\frac{x^N}{N!}} =
  \W{1}{\ldots}{\frac{x^{k-1}}{(k-1)!}} \cdot
  \W{x}{\ldots}{\frac{x^{N-k}}{(N-k)!}},
\end{equation}
where the first factor in the r.h.s.\ of~\eqref{WronskDecompose}
equals~$1$ and has degree~$0$. Denote by~$W_m$ the
second factor, the determinant of the $(N-k)\times(N-k)$ matrix with $m\equiv N-k$. 
We claim that $W_m$~is a monomial: $\deg W_m=m$; we prove
this 
by induction on $m\equiv N-k$. For $m=1$, $\deg\det(x)=1=m$. Let
$m>1$; the decomposition of~$W_m$ w.r.t.\ the last row gives
\begin{equation}\label{DecompositionWRTLastRow}
W_m=\W{x}{{\ldots}}{\frac{x^m}{m!}}=
    x\cdot\W{x}{{\ldots}}{\frac{x^{m-1}}{(m-1)!}} -
    \W{x}{{\ldots},\frac{x^{m-2}}{(m-2)!}}{\frac{x^m}{m!}},
\end{equation}
where the degree of the first Wronskian in r.h.s.\ of~(\ref{DecompositionWRTLastRow}) is $m-1$ by the inductive assumption.
Again, decompose the second Wronskian in r.h.s.\ of~(\ref{DecompositionWRTLastRow}) w.r.t.\ the last row and proceed iteratively by using the induction hypothesis.
We obtain the recurrence relation
\begin{equation}\label{RecurrentWronsk}
W_m=\sum\limits_{\ell=1}^{m-1}
W_{m-\ell}\cdot (-1)^{\ell+1} \frac{x^\ell}{\ell!} -
(-1)^m\,\frac{x^m}{m!},\qquad m\geqslant 1,
\end{equation}
whence $\deg W_m=m$.
We see that the initially taken Wronskian~(\ref{WronskDecompose}) itself is a
monomial of degree $m=N-k$ with yet unknown coefficient.

Now, we calculate the coefficient $W_m(x)/x^m\in\Bbbk$ in 
Wronskian determinant~\eqref{WronskDecompose}. Consider the generating function
\begin{equation}\label{GenFunction}
f(x)\equiv\sum\limits_{m=1}^\infty W_m(x)
\qquad \text{such that} \quad
W_m(x)=\frac{x^m}{m!}\,\frac{d^mf}{dx^m}(0),\qquad 1\leqslant m\in\BBN.
\end{equation}
Recall that $\exp(x)\equiv\sum_{m=0}^\infty x^m/m!$;
viewing~\eqref{GenFunction} as the formal sum of equations~\eqref{RecurrentWronsk}, we have
\[
f(x)=f(x)\cdot(\exp(-x)-1)-\exp(-x)+1,
\qquad \text{whence}\quad
f(x)=\exp(x)-1.
\]
Hence the required coefficient equals~$1/m!$. The proof is complete.
\end{proof}

\section{Strong homotopy deformation of the Witt algebra by Wronskians}\label{SecWitt}\noindent%
The infinite\/-\/dimensional Witt algebra of holomorphic vector fields on $\BBC\setminus \{ 0 \}$, defined by the relations $[a_i, a_j]=(j-i)\,a_{i+j}$ for $i,j\in\BBZ$, is the Virasoro algebra with zero central charge.
We now study its $L_\infty$-{}, yet in fact a strong homotopy deformation by using Wronskians. 
In the Witt algebra itself, we have the binary bracket ($N=2$) 
of the polynomial coefficient generators $a_i=x^{i+1}$, where $x\in\Bbbk$ and~$i\in\BBZ$. 

For $N\geqslant 2$, the proper choice of 
index shift in the set of generators is
$a_i=x^{i+N/2}$. 
We postulate the Wronskian determinant $W^{0,1,\ldots,N-1}$ be the $N$-ary bracket:
\begin{equation}\label{WittGeneral}
\bigl[a_{i_1},\ldots,a_{i_N}\bigr]_N=\Omega(i_1,\ldots,i_N)\,a_{i_1+\cdots+i_N};
\end{equation}
the structure constants $\Omega(i_1,\ldots,i_N)$ are totally
antisymmetric w.r.t.\ their arguments. 
Let us calculate the function~$\Omega$.

\begin{theor}\label{VanderAsCoeff}
Let $\nu_1$, ${\ldots}$, $\nu_N\in\Bbbk$ be constants and set
$\nu=\sum_{i=1}^N\nu_i$\textup{;} then we have that
\begin{equation}\label{Vander}
W^{0,1,\ldots,N-1}(x^{\nu_1},\ldots,x^{\nu_N})=
\prod\limits_{1\leqslant i<j\leqslant N}(\nu_j-\nu_i)\cdot
x^{\nu-{N(N-1)}/{2} },
\end{equation}
i.e.\ the Wronskian determinant of monomials itself is a monomial,
and its coefficient is the Vandermonde determinant.
\end{theor}

\begin{proof}
Consider determinant~\eqref{Vander}:
$A=\det\|a_{ij}\,x^{\nu_j-i+1}\|$. From $j$th column take the monomial
$x^{\nu_j-N+1}$ out of the determinant:
\[
A=x^{\nu-N(N-1)}\cdot\det\|a_{ij}\,x^{N-i}\|;
\]
all rows acquire common degrees in $x$: $\deg$(any element in $i$th
row) ${}=N-i$. From $i$th row take this common factor $x^{N-i}$ out of
the determinant:
\[
A=x^{\nu-N(N-1)/2}\cdot\det\|a_{ij}\|,
\]
where the coefficients $a_{ij}$ originate from the initial derivations
in a very special way:
for any~$i$ such that $2\leqslant i\leqslant N$, we have
\[
a_{1j}=1\qquad\text{and}\qquad
a_{ij}=(\nu_j\underline{{}-i+2})\cdot a_{i-1,j}
\quad\text{for $1<i\leqslant N$.}
\]
The underlined summand does not depend on~$j$; hence for any $k=N$,\ %
${\ldots}$,\ $2$ the determinant $\det{}\|a_{ij}\|$ can be split 
into the sum:
\begin{multline*}
\det\|a_{ij}\|=\det\|a'_{kj}=\nu_j\cdot a_{k-1,j};\quad
   a'_{ij}=a_{ij} \text{ if }i\neq k\| +
 {}\\
    {}+  \det\|a''_{kj}=(2-i)\cdot a_{k-1,j};\quad
   a''_{ij}=a_{ij} \text{ if }i\neq k\|,
\end{multline*}
where the last determinant is vanishing identically. 

Solving the recurrence relation $a_{ij}=\nu_j\cdot a_{i-1,j}$,
we obtain
\[
\det\|a_{ij}\|=\det\|\nu_j^{i-1}\|=\prod\limits_{1\leqslant k<\ell\leqslant N}
(\nu_\ell-\nu_k).
\]
This completes the proof.
\end{proof}

\begin{rem}\label{RemNotWrongPowers}
We have calculated the structure constants in~\eqref{WittGeneral}
by using a `wrong' 
basis $a_i'=x^i$ such that the resulting degree is not
$\sum_{k=1}^N\deg a_k'$. Nevertheless, 
we use the translation invariance of the Vandermonde determinant,
\[
\Omega(i_1,\ldots,i_N)=\Omega(i_1+\frac{N}{2},\ldots,i_N+\frac{N}{2}).
\]
The assertion 
is established. 
\end{rem}


\normalsize
\subsubsection*{Acknowledgements.}\normalsize%
The author thanks the organisers of 29th international conference on integrable systems \&\ 
quantum symmetries -- ISQS29 (held 7--11 July 2025 at CVUT Prague, CZ) for a warm atmosphere during the meeting. This work has been partially supported by the Bernoulli Institute (Groningen, NL) via project~110135.
The author thanks M.\,Kontsevich and V.\,Retakh for helpful discussions and advice.

\large
\appendix
\section{The conformal weight of the Wronskian determinant%
}\label{AppWronsk}\noindent%
Let us recall the behaviour of Wronskian determinants 
w.r.t.\ coordinate changes~$y=y(x)$.

\begin{theor}\label{ConformalWeightTh}
Let $\phi^i(y)$ be smooth functions for $1\leqslant i\leqslant N$, that is, 
$\phi^i$~be a scalar field of conformal weight~$0$ so that $\phi^i$
is transformed by the rule $\phi^i(y)\mapsto\phi^i(y(x))$ under a change
$y=y(x)$.
Then the transformation law for the Wronskian is 
\[
{\det\left\|\frac{d^j\phi^i}{dx^j}\right\|}_{
 \begin{array}{rcl}
    i&\!=\!&1,\ldots,N\\
    j&\!=\!&0,\ldots,N-1\\
    \phi^i&\!=\!&\phi^i(y(x))
 \end{array}
 }
= {\left(\frac{dy}{dx}\right)}^{\Delta(N)}\,
{\det\left\|\frac{d^j\phi^i}{dy^j}\right\|}_{
 \begin{array}{rcl}
    \phi^i&\!=\!&\phi^i(y).\\
    y&\!=\!&y(x)
 \end{array}
 }
\]
The conformal weight $\Delta(N)$ of the Wronskian
determinant for $N$~scalar fields~$\phi^i$, themselves of weight~$0$, is
$\Delta(N)={N(N-1)}/{2}$.
\end{theor}

\begin{proof}
Consider a function $\phi^i(y(x))$ and apply the 
$j$th power $(\Id/\Id x)^j$ of derivative~$\dd_x$  
by using the chain rule. The result is
\[
\frac{d^j\phi^i}{dy^j}\cdot{\left(\frac{dy}{dx}\right)}^j +
\text{terms with lower order derivatives }\frac{d^{j'}}{dy^{j'}},\quad
j'<j.
\]
These lower\/-\/order terms differ from the leading terms in
$(\Id/\Id x)^{j'}\,   
\phi^i(y(x))$ with $0\leqslant j'<j$ by the factors common for all~$i$; 
those lower\/-\/order terms 
produce no effect since a determinant with
coinciding (or proportional) lines equals zero.
From $i$th row of the Wronskian we extract $(i-1)$th power of $dy/dx$,
their total number being $N(N-1)/2$. This is the conformal weight
by definition.
\end{proof}

\end{document}